\documentclass{article}
 
\usepackage{amsmath,amssymb,amsfonts}
\usepackage{pstricks, pst-node}
\usepackage{bonn2007def}
\usepackage[all]{xy}
\usepackage[dvips]{graphics}

\begin{document}

\title{From quantum electrodynamics to posets of planar binary trees}
 
\author{
Fr\'ed\'eric Chapoton and Alessandra Frabetti \\
Universit{\'e} de Lyon ;\\
Universit{\'e} Lyon 1 ;\\
CNRS, UMR5208, Institut Camille Jordan,\\
43 blvd du 11 novembre 1918,\\
F-69622 Villeurbanne-Cedex, France\\
{\tt chapoton@math.univ-lyon1.fr, frabetti@math.univ-lyon1.fr}}
\date{\today}
 
\maketitle


\begin{abstract}
This paper is a brief mathematical excursion which starts from 
quantum electrodynamics and leads to the M\"obius function of the 
Tamari lattice of planar binary trees, within the framework of groups 
of tree-expanded series. 

First we recall Brouder's expansion of the photon and the electron Green's 
functions on planar binary trees, before and after the renormalization. 
Then we recall the structure of Connes and Kreimer's Hopf algebra of 
renormalization in the context of planar binary trees, and of their dual 
group of tree-expanded series. Finally we show that the M\"obius function 
of the Tamari posets of planar binary trees gives rise to a particular 
series in this group. 
\end{abstract}


\section*{Introduction}

Planar binary trees are among the most classical objects in combinatorics. 
Being counted by the Catalan numbers, they are in bijection with more than one hundred other combinatorial objects, cf.\cite{Stanley2}, 
such as noncrossing partitions and Dyck paths. 
In the last fifteen years, they have begun to play a key role in some 
algebraic structures, ranging from groups to Hopf algebras and to
operads, cf. \cite{LodayDendriform,Hivert3,Aguiar_Sottile}.
In this article, we would like to show how a problem in quantum
electrodynamics can lead to such algebraic structures. 

As shown by Christian Brouder in \cite{Brouder}, in quantum electrodynamics 
the Green functions can be described perturbatively with an expansion
over planar binary trees. In the quantum field theory describing a
scalar self-interacting field, the Green functions are naturally
expanded over rooted trees (non-planar and non-binary). In contrast to
this, in quantum electrodynamics Brouder adopted binary trees to
describe the coupling of two interacting fields, the left path for the
electron field and the right path for the photon field, and adopted 
planar trees to respect the non-commutativity of the product among the
Green functions' coefficients, which are $4\times 4$ matrices. 

The expansion of the Green functions over planar binary trees requires
some simple grafting operations among trees, \cite{Brouder}. 
Instead, the renormalization of these Green functions gives rise to 
some Hopf algebras on planar binary trees, cf. \cite{BFqedren}, 
analogue to the famous Connes-Kreimer Hopf algebra encoding the 
renormalization of the scalar $\Phi^3$ theory, cf. \cite{Kreimer,CKI}. 
These Hopf algebras can naturally be seen as coordinate rings of some
pro-algebraic groups of formal series expanded over trees. 
One of them turns out to be a related to an operad on planar binary 
trees, cf. \cite{Frabetti2008}.

There are in fact two operads on planar binary trees in the
literature. In the present context we are concerned with the duplicial
operad, cf. \cite{Loday2006}. The other one is the dendriform operad
\cite{LodayDendriform}, which has been the subject of a lot of
attention recently. In particular, it has been shown that the
dendriform operad is deeply related to the family of posets called
Tamari lattices, cf. \cite{Tamari}. In this article, we show that
Tamari lattices are also directly related to the duplicial operad. To
do it, we compute the inverse of some series in the group of
tree-expanded series associated with the duplicial operad.  \bigskip

The paper is organized as follows. In the first section we recall 
Brouder's expansion of the photon and the electron Green's functions 
on planar binary trees, before and after the renormalization. 
In the second section we review the Hopf algebras of renormalization
for quantum electrodynamics in the context of planar binary trees, 
and their dual groups of tree-expanded series. 
In the last section we show that the M\"obius function of the Tamari 
posets of planar binary trees gives rise to a particular series in the
group of tree-expanded series which generalises the group of formal
diffeomorphisms on a line. 
\bigskip\bigskip  

\noindent
{\bf Acknowledgements.} 
This work was partially financed by the ANR-06-BLAN-0380 ``HopfCombOp'' and by the 
ANR-05-JCJC-0044-01 ``AHBE''.


\tableofcontents


\section{Quantum electrodynamics and tree-expanded 
series}

Quantum electrodynamics (QED) is the quantisation of classical 
electromagnetism, that is the field theory describing the 
attraction/repulsion between electrons, by means of the photons as
mediators. 
In this section we review the perturbative expansion of the 2-point 
Green's functions of the electron and the photon on planar binary
trees, as proposed by C.~Brouder in \cite{Brouder}.


\subsection{Dyson-Schwinger equations for the electron and the photon 
propagators}

Denote by $\R^{1,3}$ the Minkowski space, with flat diagonal metric 
$g_{\mu\nu}=(1,-1,-1,-1)$, relativistic space-time coordinates 
$x=x^\mu \in \R^{1,3}$ and momenta $p=p_\mu$. 

Let $\psi:\R^{1,3}\longrightarrow\C^4$ be the fermionic field
describing the electron, with mass $m$ and electric charge $e$, 
and let $A=(A^\mu):\R^{1,3} \longrightarrow \C^4$ be the massless
bosonic field describing the electromagnetic potential (the photon). 
The dynamic of the system of interacting electrons and photons is
described by the Lagrangian density 
\begin{align*}
\L_{QED}(\psi,A;e,m) &= \L_{Dirac}(\psi;m) + \L_{Maxwell}^{\xi}(A^\mu) 
+ \L_{int}(\psi,A^\mu;e) 
\end{align*}
where 
\begin{align*}
\L_{Dirac}(\psi;m) 
&= \bar{\psi}(x^\mu)\big(i\gamma^\mu \partial_\mu -m\big)\psi(x^\mu), 
\end{align*}
is the Lagrangian describing the free electrons 
($\gamma^\mu$ are the Dirac 4x4 matrices and $\bar{\psi}=\psi^*\gamma^0$ 
is the anti-fermionic field, e.g. the positron), 
\begin{align*}
\L_{Maxwell}^{\xi}(A) 
&= -\frac{1}{4}\big(\partial^\mu A^\nu-\partial^\nu A^\mu\big)^2 
-\frac{\xi}{2}\big(\partial_\nu A^\mu\big)^2 , 
\end{align*}
is the Lagrangian describing the free photons ($\xi$ is a gauge parameter 
introduced as a fictive mass to avoid the problems due to the absence 
of the photon mass), and 
\begin{align*}
\L_{int}(\psi,A^\mu;e) &= -e\ \bar{\psi} \gamma^\mu \psi A^\mu
\end{align*}
is the term describing the interaction (of order 3), with coupling
constant given by the electric charge $e$. 

In this paper we are concerned with the the connected 2-point Green functions
\begin{align*}
\left\{\begin{array}{rl}
S(x,y) &\ds = \la \psi(x)\psi(y)\ra, \qquad\mbox{for the electron} \\ &\\ 
D_{\mu\nu}(x,y) &\ds = \la A_\mu(x) A_\nu(y)\ra, \qquad\mbox{for the photon} 
\end{array}\right. 
\end{align*}
whose square module represent the probability that the quantum field 
moves from the point $y$ to the point $x$ in space-time.
(We suppose that the 1-point Green's functions are zero.)  
If the fields are free, that is, the Lagrangian of the system is the sum 
$\L_{Dirac}(\psi;m)+\L_{Maxwell}^{\xi}(A^\mu)$ without the interaction
part, the 2-point Green functions coincide with the free propagators. 
These are $4\times 4$ matrix-valued distributions on $\R^{1,3}$, whose
Fourier transforms on the momenta space are 
\begin{align*}
\left\{\begin{array}{rl}
S_0(p) &\ds = \frac{1}{\gamma^\mu p_\mu -m +i\epsilon}, \\ &\\ 
D_{0,\mu\nu}(p) &\ds = -\frac{g_{\mu\nu}}{p^2+i\epsilon} 
+ \left(1-\frac{1}{\xi}\right) 
\frac{p_\mu p_\nu}{(p^2+i\epsilon)^2}. 
\end{array}\right. 
\end{align*}
 
If the fields interact according to the Lagrangian $\L_{QED}$, 
the 2-point Green functions (that we shall keep calling propagators) 
satisfy a system of two functional equations called {\em Dyson-Schwinger 
equations\/}, analogue to the equation of motion in the case of 
classical fields. 
For an isolated system, C.~Brouder presented these two equations 
in the following form, cf. \cite{Brouder}: 
\begin{align}
\label{DS-QED}
\left\{\begin{array}{rl}
S(p) &\ds = S_0(p) +i\ e^2\ S_0(p) \int\frac{d^4q}{(2\pi)^4}\ \gamma^\lambda\ 
D_{\lambda\lambda'}(q) \frac{\delta S(p-q)}{e\ \delta A_{\lambda'}(q)}, \\ &\\ 
D_{\mu\nu}(p) &\ds = D_{0,\mu\nu}(p) -i\ e^2\ D_{0,\mu\lambda}(p) 
\int\frac{d^4q}{(2\pi)^4}\ \mathrm{tr}\left[\gamma^\lambda\ 
\frac{\delta S(q)}{e\ \delta A_{\lambda'}(-p)} \right]\
D_{\lambda'\nu}(p). 
\end{array}\right. 
\end{align}
In this equations, the electron propagator $S(p)$ depends {\em a
priori} on an external electromagnetic field $A_{\lambda}(q)$, and the
functional derivative $\frac{\delta S(p)}{\delta A_{\lambda}(q)}$ 
detects this dependence. Then, since the equations are valid for an
isolated system, we imply that $S(p)$ and its functional derivatives 
are evaluated at $A_{\lambda}(q)=0$. 

The solution of the equations (\ref{DS-QED}) is a perturbative series in 
the powers of the {\em fine structure constant $\alpha=\frac{e^2}{4\pi}$}, 
\begin{align}
\label{QED-series-integer}
\left\{\begin{array}{rl}
S(p) &\ds = S_0(p) +\sum_{n\geq 1} \alpha^n\ S_n(p) , \\ & \\ 
D_{\mu\nu}(p) &\ds 
= D_{0,\mu\nu}(p) +\sum_{n\geq 1} \alpha^n\ D_{n,\mu\lambda}(p),
\end{array}\right. 
\end{align}
obtained recursively from the initial datas, 
which are the free propagators $S_0(p)$, $D_{0,\mu\nu}(p)$ and all the 
functional derivatives $\frac{\delta S_0(p)}{\delta A_{\lambda}(q)}$, 
$\frac{\delta^2 S_0(p)}{\delta A_{\lambda}(q_1) \delta A_{\lambda}(q_2)}$, 
etc. 
At each new order of perturbation, the coefficient of the series is 
an analytical expression involving the previous coefficients. 
In fact, for the series (\ref{QED-series-integer}), the system 
(\ref{DS-QED}) is equivalent to 
\begin{align}
\label{QED-sol-integer}
\left\{\begin{array}{rl}
S_n(p) &\ds 
= i S_0(p) \sum_{k+l=n-1}  \int\frac{d^4q}{(2\pi)^4}\ \gamma^\lambda\ 
D_{k,\lambda\lambda'}(q) \frac{\delta S_l(p-q)}{e\ \delta A_{\lambda'}(q)}, 
\\ &\\ 
D_{n,\mu\nu}(p) &\ds = -i D_{0,\mu\lambda}(p) \sum_{k+l=n-1} 
\int\frac{d^4q}{(2\pi)^4}\ \mathrm{tr}\left[\gamma^\lambda\ 
\frac{\delta S_k(q)}{e\ \delta A_{\lambda'}(-p)} \right]\
D_{l,\lambda'\nu}(p). 
\end{array}\right. 
\end{align}


\subsection{Renormalization factors and Feynman graphs}

Equations (\ref{QED-sol-integer}) can be formally solved, but they
present a major problem of quantum field theory: at any order 
of perturbation $n>0$, even for $n=1$, the perturbative coefficients 
$S_n(p)$ and $D_{n,\mu\nu}(p)$ contain divergent integrals. 
The physical explanation is that the quantum fields $\psi$ and $A$, 
and the measurable parameters $e$ and $m$, have an energy 
level which is far too high with respect to the classical ones, and  
the classical Lagrangian $\L_{QED}(\psi,A;e,m)$ is not sufficient to 
describe them. As a consequence, the Lagrangian must be
modified into the so-called the {\em renormalized Lagrangian} 
\begin{align*}
\L_{QED}^{ren}(\psi,A;e,m) 
&= Z_2(e)\ \bar{\psi}(x^\mu)\big(i\gamma^\mu \partial_\mu 
-(m+\delta m(e)\big)\psi(x^\mu) \\ 
&\quad - Z_3(e)\ \frac{1}{4}\big(\partial^\mu A^\nu-\partial^\nu A^\mu\big)^2
-\frac{\xi}{2}\big(\partial_\nu A^\mu\big)^2 \\ 
&\quad -Z_1(e)\ e\ \bar{\psi} \gamma^\mu \psi A^\mu, 
\end{align*}
where the so-called {\em renormalization factors} $Z_1(e)$, $Z_2(e)$, 
$Z_3(e)$ and $\delta m(e)$ are series in the powers of $\alpha$ which 
compensate the divergences and allow to produce finite quantities.   
If we call {\em bare parameters} the rescaled parameters 
\begin{align}
\label{bare parameters}
\left\{\begin{array}{rl}
m_0 &= m+ \delta m(e), 
\\ & \\ 
e_0 &= \dfrac{e\ Z_1(e)}{Z_2(e) \sqrt{Z_3(e)}}, 
\end{array}\right. 
\end{align}
and {\em bare fields} the rescaled fields 
\begin{align*}
\left\{\begin{array}{rl}
\psi_0 &= \sqrt{Z_2(e)}\ \psi, 
\\ & \\ 
A^\mu_0 &= \sqrt{Z_3(e)}\ A^\mu, 
\end{array}\right. 
\end{align*}
it turns out that 
\begin{align*}
\L_{QED}^{ren}(\psi,A;e,m) &= \L_{QED}(\psi_0,A_0;e_0,m_0). 
\end{align*}
Then, using the Dyson-Schwinger equations (\ref{DS-QED}) 
one can compute the {\em bare propagators} $S(p;e_0,m_0)$ and 
$D_{\mu\nu}(p;e_0,m_0)$ relative to the bare fields, and F.~Dyson 
showed in \cite{Dyson} that the {\em renormalized propagators} 
can be found simply as 
\begin{align}
\label{Dyson Formula}
\left\{\begin{array}{rl}
S^{ren}(p;e,m) &= Z_2(e)^{-1}\ S(p;e_0,m_0), 
\\ & \\ 
(D_{\mu\nu}^T)^{ren}(p;e,m) &= Z_3(e)^{-1}\  D_{\mu\nu}^T(p;e_0,m_0), 
\end{array}\right. 
\end{align}
where $D_{\mu\nu}^T(p) = D_{\mu\nu}(p)-\frac{1}{\xi} 
\frac{p_{\lambda} p_{\mu}}{(p^2+ i\epsilon)^2}$ is the transversal part 
of the propagator, while the longitudinal part $\frac{1}{\xi} 
\frac{p_{\lambda} p_{\mu}}{(p^2+ i\epsilon)^2}$ is not affected by the 
renormalization. 

One way to compute the renormalization factors is to expand the propagators 
over all possible connected {\em Feynman diagrams} with 2 external legs. 
The amplitude of the graphs are the simplest integrals appearing in 
the perturbative coefficients of the integral expansion of the propagators. 
Those which diverge must be renormalized. The divergence extracted from 
a graph forms the {\em counterterm} of the graph, and the renormalization 
factors are constructed by reassembling the counterterms of all concerned 
graphs. 

Feynman graphs describe all possible virtual interactions between electrons 
and photons in the indeterminate quantum fluctuation. 
They are constructed on some fixed types of edges, which represent the free 
propagators, and with some fixed types of vertices, which represent the 
allowed interactions, and therefore are completely determined by the 
Lagrangian. If we denote by $\Gamma$ the connected Feynman graphs with 2 
external legs, and we distinguish the electron graphs ($e$) from the photon 
graphs ($\gamma$), the perturbative expansion of the QED propagators is 
\begin{align}
\label{QED-series-graphs}
\left\{\begin{array}{rl}
S(p) &\ds = S_0(p) +\sum_{\Gamma^e} \alpha^{L(\Gamma)}\ S_{\Gamma}(p),
\\ & \\ 
D_{\mu\nu}(p) &\ds 
= D_{0,\mu\nu}(p) +\sum_{\Gamma^\gamma} \alpha^{L(\Gamma)}\ 
D_{\Gamma,\mu\lambda}(p),
\end{array}\right. 
\end{align}
where $L(\Gamma)$ denotes the number of loops of the graph. 
The coefficients $S_{\Gamma}(p)$ and $D_{\Gamma,\mu\lambda}(p)$ are some 
analytical expressions called the {\em amplitude\/} of the graph. 
The integral expansion (\ref{QED-series-integer}) and the diagrammatic
expansion (\ref{QED-series-graphs}) of the solution of Eqs.~(\ref{DS-QED}) 
are of course related: at each order of perturbation we have 
\begin{align*}
\left\{\begin{array}{rl}
S_n(p) &\ds = \sum_{L(\Gamma)=n}\ S_{\Gamma}(p), \\ &\\ 
D_{n,\mu\nu}(p) &\ds = \sum_{L(\Gamma)=n}\ D_{\Gamma,\mu\lambda}(p). 
\end{array}\right. 
\end{align*}

The system~(\ref{DS-QED}), applied to the series expanded as 
(\ref{QED-series-graphs}), gives an algorithm to construct the amplitudes, 
called the {\em Feynman's rules\/}, cf. for instance Appendix A in 
\cite{ItzyksonZuber}. 
Roughly speaking, Feynman's rules are a dictionary between graphical signs 
(edges, vertices, loops) and analytical expressions (free propagators,
factors of the coupling constant, integrals, overall symmetry factors). 
Feynman's rules have many advantages: they are intuitive and easy to 
memorize, they allow to compute the coefficients of the propagators 
each independently of any other, and finally they allow to compute the
counterterms of the graphs and therefore the coefficients of the
renormalization factors.  
Based on this method, one can prove for instance the {\em Ward-Takahashi 
identity} $Z_2(e)=Z_1(e)$, which imply that the charge is renormalized as 
\begin{align}
\label{Ward-Takahashi}
e_0 &= e\ Z_3(e)^{-\frac{1}{2}}.  
\end{align}
The computation of the counterterms requires a specific algorithm, 
the so-called {\em BPHZ formula}, which takes into account the 
subdivergences of the complicated graphs. 
The intricate combinatorics of this algorithm is nowadays completely 
clear, since A.~Connes and D.~Kreimer proved in \cite{CKI,CKII}
that it is equivalent to a Hopf algebra structure on the set of
Feynman graphs. 

However the method based on Feynman graphs has also two main disadvantages. 
On one side, the number of Feynman diagrams increases exponentially with 
the number of vertices, making difficult to keep control of the summability of
the perturbative expansions on graphs. 
On the other side, the amplitude of a graph with many vertices turns out 
to be a very complicated analytical expression. The computation of the
propagators requires hundreds of integrals even at low orders of
perturbation, and the computation of the counterterms, performed
through the BPHZ formula, makes the computation of each renormalized
integral even longer. 


\subsection{Tree-expansion of QED propagators}

In the paper \cite{Brouder}, C.~Brouder proposed an alternative perturbative 
solution of the Dyson-Schwinger equations~(\ref{DS-QED}), which has
the main advantage of reducing drastically the number of computations
required at each order of perturbation. To do it, he based the 
expansion of the perturbative series on the set of {\em planar binary 
rooted trees\/}, instead of Feynman graphs. 
These trees are planar graphs without loops, and a preferred external
edge called the root. For a tree $t$, we denote by $|t|$ the number of
its internal vertices. The tree $\treeO$, with no internal vertices, 
is called the {\em root tree\/}. The tree $\treeA$, with one internal 
vertex, is called the {\em vertex tree\/}. Other examples of trees, 
those with number of internal vertices equal to two and three, are 
\begin{align*}
\treeAB, \quad \treeBA, \quad 
\treeABC, \quad \treeBAC, \quad \treeACA, \quad \treeCAB, \quad \treeCBA. 
\end{align*}
Let us denote by $Y_n$ the set of planar binary trees with $n$ internal 
vertices, by $Y=\bigcup_{n=0}^\infty Y_n$ the set of all trees, and by 
$\Y=\bigcup_{n=1}^\infty Y_n$ the set of all trees with at least one 
internal vertex. 
Finally, let us denote by $\vee :Y_n \times Y_m \longrightarrow Y_{n+m+1}$ 
the operation which grafts two trees on a new root, that is
\begin{align*}
s \vee t &= \lrgraft{s}{t}. 
\end{align*}

If we suppose that the bare 2-point Green's functions of QED are 
formal series of the form 
\begin{align}
\label{QED-series-trees}
\left\{\begin{array}{rl}
S(p;e_0,m_0) &\ds = S_0(p) +\sum_{t\in\Y} e_0^{2|t|}\ S_{t}(p) , \\ & \\ 
D_{\mu\nu}(p;e_0,m_0) &\ds 
= D_{0,\mu\nu}(p) +\sum_{t\in\Y} e_0^{2|t|}\ D_{t,\mu\lambda}(p),
\end{array}\right. 
\end{align}
then the system~(\ref{DS-QED}) determines the coefficients of these series 
as 
\begin{align}
\label{QED-sol-trees}
\left\{\begin{array}{rl}
S_t(p) &\ds 
= i S_0(p) \int\frac{d^4q}{(2\pi)^4}\ \gamma^\lambda\ 
D_{t^l,\lambda\lambda'}(q) 
\frac{\delta S_{t^r}(p-q)}{e_0\ \delta (A_0)_{\lambda'}(q)}, 
\\ &\\ 
D_{t,\mu\nu}(p) &\ds = -i D_{0,\mu\lambda}(p) 
\int\frac{d^4q}{(2\pi)^4}\ \mathrm{tr}\left[\gamma^\lambda\ 
\frac{\delta S_{t^l}(q)}{e_0\ \delta (A_0)_{\lambda'}(-p)} \right]\ 
D_{t^r,\lambda'\nu}(p),
\end{array}\right. 
\end{align}
where, for any tree $t\neq \treeO$, the trees $t^l$ and $t^r$ are such that 
$t=t^l\vee t^r$. 
The expansion on trees (\ref{QED-series-trees}) is related to the
integral expansion (\ref{QED-series-integer}) because the coefficients
of the integral series at the perturbative order $n$ are the sum of 
the corresponding coefficients over the trees $t$ such that $2|t|=n$. 
It is also related to the expansion (\ref{QED-series-graphs}) on Feynman
graphs because the coefficients depending on a tree $t$ are the sum of
the coefficients depending on some Feynman graphs which can be found
with an algorithm given in \cite{BFqedren}. Roughly speaking, a tree
corresponds to a finite sum of Feynman diagrams. 

{\em A priori}, a BPHZ formula for the amplitude of the trees is not
known, nor its equivalent algebraic version based on a Hopf algebra of
planar binary trees. 
Therefore, to find the tree-expansion for the renormalized propagators, 
it is necessary to solve directly the Dyson-Schwinger equations to which
they are constrained. These are best given on the inverse of the photon 
propagator, that is, on the {\em vacuum polarization} 
\begin{align*}
\Pi_{\lambda\mu}(p) & =  
(p_\lambda p_\mu -p^2 g_{\lambda\mu})-\xi\ p_\lambda p_\mu
-{[D^{-1}]}_{\lambda\mu}(p).
\end{align*}
Then, the {\em renormalized Dyson-Schwinger equations} are given by 
\begin{align}
\label{DS-QED-REN}
\left\{\begin{array}{rl}
S^{ren}(p;e,m) &\ds = Z_2^{-1}\ S_0(p)- \delta m\ S_0(p)\ S^{ren}(p) 
+ i e^2 S_0(q) \int\frac{\dd^4q}{(2\pi)^4} 
\gamma^\lambda {D^{ren}}_{\lambda\lambda'}(q)
\frac{\delta S^{ren}(p-q)}{e\delta A_{\lambda'}(q)}, \\ &\\ 
\Pi^{ren}_{\lambda\mu}(p;e,m) 
&\ds = (1-Z_3)(p_\lambda p_\mu -p^2 g_{\lambda\mu})
- i\ e^2\ Z_2 \int \frac{\dd^4q}{(2\pi)^4} 
\mathrm{tr}\left[\gamma_\lambda 
\frac{\delta S^{ren}(q)}{e \delta A^\mu(-p)}\right]. 
\end{array}\right. 
\end{align}
Assuming that the propagators are expanded on planar binary trees
according to 
\begin{align}
\label{QED-series-trees-ren}
\left\{\begin{array}{rl}
S^{ren}(p;e,m) &\ds = S_0(p) +\sum_{t\in\Y} e^{2|t|}\ S^{ren}_{t}(p) , \\ & \\ 
D^{ren}_{\mu\nu}(p;e,m) &\ds 
= D_{0,\mu\nu}(p) +\sum_{t\in\Y} e^{2|t|}\ D^{ren}_{t,\mu\lambda}(p),
\end{array}\right. 
\end{align}
and that the renormalization factors are also expanded on planar
binary trees, that is, 
\begin{align}
\nonumber
Z_2(e) &= 1+\sum_{|t|>0} e^{2|t|}\ C_2(t),
  \quad\mathrm{with}\quad C_2(\treeO)=1, \\
\label{Zfactors-trees}
Z_3(e) &= 1-\sum_{|t|>0} e^{2|t|}\ C_3(t),
  \quad\mathrm{with}\quad C_3(\treeO)=1, \\
\nonumber
\delta m(e) &= \sum_{|t|>0} e^{2|t|}\ C_m(t),
  \quad\mathrm{with}\quad C_m(\treeO)=0, 
\end{align}
then the system~(\ref{DS-QED-REN}) determines the coefficients of the
renormalized propagators as 
\begin{align}
\label{QED-sol-trees-REN}
\left\{\begin{array}{rl}
S^{ren}_t(p) &\ds 
= C_2(I(t)) - S_0(p)\ (C_m\ast S^{ren})(t) 
+ i S_0(p) \int\frac{d^4q}{(2\pi)^4}\ \gamma^\lambda\ 
D^{ren}_{t^l,\lambda\lambda'}(q) 
\frac{\delta S^{ren}_{t^r}(p-q)}{e\ \delta A_{\lambda'}(q)}, 
\\ &\\ 
D^{ren}_{t,\mu\nu}(p) &\ds = D^{ren}_{\treeO\vee t_r,\mu\nu}(p)\ 
(p_\mu p_\nu-p^2 g_{\mu\nu}) D^{ren}_{t_l,\mu\nu}(p), 
\qquad\mbox{if $t=t_l\vee t_r$ with $t_l\neq\treeO$}, 
\\ &\\ 
D^{ren}_{\treeO\vee t,\mu\nu}(p) 
&\ds = C_3(\treeO\vee t) D^T_{0,\mu\lambda}(p)
-i D^T_{0,\mu\lambda}(p) 
\int\frac{d^4q}{(2\pi)^4}\ \mathrm{tr}\left[\gamma^\lambda\ 
\frac{\delta (S^{ren}(q)\ast C_2)(t)}{e\ \delta A_{\lambda'}(-p)} \right]\ 
D^T_{0,\lambda'\nu}(p). 
\end{array}\right. 
\end{align}
In these equations we make use of a convolution $\ast$ dual to the 
pruning coproduct on trees defined recursively by 
$\Die(\treeO)=\treeO\otimes\treeO$ and 
\begin{align}
\label{pruning coproduct} 
\Die(t) &= \treeO\otimes t 
+ \sum_{\Die(t_r)} (t_l\vee (t_r)_{(1)})\otimes (t_r)_{(2)}, 
\qquad\mbox{for $t=t_l\vee t_r$}.
\end{align}
The map $I$ appearing in the first equation is the antipode associated
to this coproduct. 

The coproduct $\Die$ turns out to be related to the operation 
{\em under} among trees, introduced by J.-L.~Loday in 
\cite{LodayArithmetree}: given two planar binary trees
$u,v\neq\treeO$, the tree $u$ under $v$, denoted by $u\under v$, is
the tree obtained by grafting the right-most leaf of $u$ under the
root of $v$, that is, 
\begin{align*}
u \under v &= \rgraft{u}{v}. 
\end{align*}
Then, Eq.~(\ref{pruning coproduct}) can be simply rewritten as 
\begin{align*}
\Die(t) &= \sum_{t=u\under v} u \otimes v, 
\end{align*} 
where we also suppose that $\Die$ is a multiplicative
map which respects the decomposition of a tree $t=t_l\vee t_r$ into 
its two multiplicative factors $t=(t_l\vee\treeO)\under t_r$. 

In fact, observing carefully the equations for the photon propagator,
one can see that another operation among trees is implicitly used: 
the decomposition of a tree $t$ into two multiplicative factors $t^l$ and 
$\treeO\vee t^r$. The product among trees which corresponds to this
decomposition is the operation {\em over},
cf.~\cite{LodayArithmetree}: 
given two planar binary trees $u,v\neq\treeO$, the tree $u$ over $v$, 
denoted by $u\over v$, is
the tree obtained by grafting the root of $u$ over the left-most leaf 
of $v$, that is, 
\begin{align*}
u \over v &= \lgraft{v}{u}. 
\end{align*}
Associated to this product there is of course a coproduct 
\begin{align*}
\Dip(t) &= \sum_{t=u\over v} u \otimes v. 
\end{align*} 

It turns out that the relationship between the bare propagators 
$S(p;e_0,m_0)$, $D_{\mu\nu}(p;e_0,m_0)$ and the renormalized
propagators $S^{ren}(p;e,m)$, $D^{ren}_{\mu\nu}(p;e,m)$, all expanded
on planar binary trees, is governed by two Hopf algebra structures
on the set of trees based on the two coproducts $\Die$, $\Dip$ 
and also on another coproduct $\Da$ analogue to the Connes-Kreimer 
coproduct on Feynman diagrams. 


\section{Renormalization Hopf algebras and groups 
of tree-expanded series}

\subsection{QED renormalization Hopf algebras on trees}

In \cite{BFqedren} and \cite{BF2003} it is shown that the relationship
between the QED propagators before and after the renormalization, excluding 
the mass renormalization and considering the expansion over planar binary 
trees, is described by a semidirect coproduct of Hopf algebras on trees, 
involving a Hopf algebra $\He$ for the electron propagator, one for the 
photon propagator $\Hp$, and one for the electric charge $\Ha$. 

As algebras, $\He$ and $\Hp$ are both isomorphic to the free 
non-commutative algebra generated by all trees, where we identify 
the root tree $\treeO$ to the formal unit. 
Therefore $\He=\Hp=\Q\langle\Y\rangle \cong \Q\langle Y \rangle/(\treeO-1)$.
The coalgebra structures are given by the pruning coproducts
$\Dip:\Hp\longrightarrow\Hp\otimes\Hp$ and
$\Die:\He\longrightarrow\He\otimes\He$ defined as the dual operations
of the two grafting products {\em over} and {\em under} on trees: 
\begin{align*}
\mbox{over:}\qquad t \over s = \lgraft{s}{t} , 
&\qquad \Dip(u) = \sum_{t\over s=u} t \otimes s, \\ 
\mbox{under:}\qquad t \under s = \rgraft{t}{s} ,
&\qquad \Die(u) = \sum_{t\under s=u} t \otimes s. 
\end{align*}
The counits $\varepsilon:\Hp\longrightarrow \Q$ and
$\varepsilon:\He\longrightarrow \Q$ are dual to the unit $\treeO$,
that is $\varepsilon(t)= \delta_{t,\treeO}$. 
Note that the Hopf algebras $\He$ and $\Hp$ are neither commutative nor 
cocommutative. 

On the other side, $\Ha$ is the abelian quotient of the algebra $\Q Y$ 
of all trees endowed with the over product (which is not commutative). 
Thus the root tree $\treeO$ is the unit, and if we set $V(t)=\treeA\under t$, 
the algebra $\Ha$ is in fact isomorphic to the polynomial algebra 
$\Q[V(t), t\in Y]$. The coproduct $\Da:\Ha \longrightarrow \Ha \otimes\Ha$ 
is defined on the generators by the assignment 
\begin{align}
\label{Da-def}
\Da (V(t)) &= 1\otimes V(t)+\da(V(t)), 
\end{align}
where $\da:\Ha\longrightarrow \Ha\otimes\Ha$ is a right coaction of $\Ha$ 
on itself (w.r.t. the coproduct $\Da$), defined recursively as 
\begin{align}
\label{da-def}
\da(V(t)) &= (V\otimes\Id) \left[\Da(t^l)\over\da(V(t^r))\right],
\end{align} 
where $t=t^l\over V(t^r)$. 
The counit $\varepsilon:\Ha\longrightarrow \Q$ is of course the unital algebra 
morphism with value $\varepsilon(V(t))=0$ on the generators. 

According to \cite{BFqedren} and \cite{BF2003}, the massless renormalization 
of the electron propagator expanded over planar binary trees is described by 
the semidirect (or smash) coproduct Hopf algebra $\Hq=\Ha\ltimes\He$, 
while the massless renormalization of the photon propagator is described by 
the Hopf algebra $\Ha$ itself, seen as a Hopf subalgebra of the semidirect 
coproduct $\Ha\ltimes\Hp$. These semidirect coproduct Hopf algebras can 
be defined, according to \cite{Molnar1977}, because there is a coaction 
of $\Ha$ on $\He$ and another one on $\Hp$, namely two variations of $\da$, 
which preserve the algebra structures. 

The results can be summarized as follows: 
\begin{itemize}
\item
the bare and renormalized amplitudes $S(p;e_0,m_0)$, $S^{ren}(p;e_0,m_0)$ 
and the renormalization factor $Z_2(e)$ can be seen as characters of
the Hopf algebra $\He$ with values in an algebra of regularized 
amplitudes $\A_{reg}$; 
\item 
the bare and renormalized amplitudes $D_{\mu\nu}(p;e,m)$, 
$D^{ren}_{\mu\nu}(p;e,m)$ and the renormalization factor $Z_3(e)$ 
can be seen as characters of the Hopf algebra $\Hp$, 
with values in the same algebra $\A_{reg}$; 
\item 
the bare charge $e_0(e)$ can be seen as a character of the Hopf algebra $\Ha$, 
with values in the same algebra $\A_{reg}$; 
\item 
the relationship between the tree-expanded coefficients of the
propagators before and after the renormalization for massless QED 
is given by the convolution associated to the coproducts of the Hopf 
algebras $\Hq=\Ha\ltimes\He$, for the electron, and 
$\Ha \subset \Ha\ltimes\Hp$ for the photon. 
\end{itemize}

It is well known that the characters of a commutative Hopf algebra form 
a group. If the algebra is the inductive limit of finite-dimensional 
graded algebras, as is the case here, the group is pro-algebraic (and even 
pro-unipotent). In the next section we describe the pro-algebraic groups 
formed by the characters of the Hopf algebras related to the QED 
renormalization, abelianized when necessary. 


\subsection{Groups of tree-expanded series}

The abelian quotients $\Hp_{ab}$ and $\He_{ab}$ of the two Hopf algebras 
$\Hp$ and $\He$ are the coordinate rings of two pro-algebraic groups,
denoted by $\Gp$ and $\Ge$ respectively, cf.~\cite{Frabetti2008}. 
For a given unital, associative and commutative algebra $\A$, 
consider the set 
\begin{align*}
\Gi_Y &= \left\{ \sA=\sum_{t\in Y} \sA_t\ t,\ \sA_{\treeO}=1 \right\}, 
\end{align*}
of tree-expanded series with coefficients $\sA_t$ in $\A$ and invertible 
constant term. 
The groups $\Gp$ and $\Ge$ are given on $\Gi$ respectively by 
the over and the under products on series, induced by the analogous 
operations on trees: 
\begin{align*}
\mbox{over}\qquad \sA\over \sB = \sum_{s,t\in Y} \sA_s\ \sB_t\ {s\over t}, 
\qquad\mbox{for $\Ge$}, \\ 
\mbox{under}\qquad \sA\under \sB = \sum_{s,t\in Y} \sA_s\ \sB_t\ {s\under t}, 
\qquad\mbox{for $\Gp$}. 
\end{align*}
In both cases, the unit is given by the series $\treeO$. 

We can identify a tree-expanded series $\sA=\sum_{t\in Y} \sA_t\ t$ with 
a ``generalized'' series on a variable $x$, by setting 
$\sA(x)=\sum_{t\in Y} \sA_t\ x^t$ , where the monomial $x^t$ is a
formal symbol.
Then, it is easy to see that the groups $\Ge$ and $\Gp$ are 
non-abelian generalizations of the abelian group 
\begin{align*}
\Gi &= \left\{ \sA(x)=\sum_{n\geq 0} \sA_n\ x^n,\ \sA_0=1 \right\}, 
\end{align*}
of usual invertible series in one variable, with coefficients $\sA_n$ in
$\A$, and considered with the multiplication. 
\bigskip 

Similarly, the Hopf algebra $\Ha$ can be realized as a group of 
tree-expanded series, endowed, this time, with an operation which 
generalizes the composition (substitution) of formal series. 
To do this, we first consider the set 
\begin{align*}
\Gd_Y &= \left\{ \sA=\sum_{t\in \Y} \sA_t\ t, 
\ \sA_{\treeA}=1 \right\} 
\end{align*}
of tree-expanded formal diffeomorphisms with coefficients $\sA_t$ in $\A$. 
The composition of two tree-expanded series 
$\sA=\sum_{t} \sA_t\ t$ and $\sB=\sum_{s} \sB_s\ s$ 
is given by 
\begin{align*}
\sA \circ \sB
& = \underset{s_1,s_2,\dots,s_{|t|}\in \Y}{\sum_{t\in \Y}} 
\sA_t\ \sB_{s_1} \sB_{s_2}\dots \sB_{s_{|t|}} \ 
{\mu_t(s_1,s_2,\dots,s_{|t|})},  
\end{align*}
where $\mu_t(s_1,s_2,\dots,s_{|t|})$ is the tree obtained by inserting
the trees $s_1$, \dots, $s_{|t|}$ in the vertices of $t$, ordered from
left to right and from the leaves to the root of $t$. 
In fact, the operation $\mu$ is the operadic composition of the
Duplicial operad, also called OverUnder operad, cf. 
\cite{Loday2006,AguiarLivernet,VDLaan}, and the group $\Gd_Y$ is an
example of a general construction which works for any connected operad, cf. 
\cite{Chapoton,VanDerLaan,Frabetti2008}.  
The unit of the group $\Gd_Y$ is the series $\treeA$. 

As before, we can identify a tree-expanded series 
$\sA=\sum_{t\in \Y} \sA_t\ t$ with a ``generalized'' series 
in one variable, $\sA(x)=\sum_{t\in \Y} \sA_t\ x^t$. 
Then, if we define the power of the series $\sB(x)$ by a tree 
$t\neq\treeO$ as $\sB(x)^t = \mu_t(\sB(x),\sB(x),\dots,\sB(x))$, 
the composition of tree-expanded series can also be seen as 
a substitution, that is $(\sA \circ \sB)(x) = \sA(\sB(x)) 
= \sum_{t\in \Y} \sA_t\ \sB(x)^t$. 
In \cite{Frabetti2008} it is then shown that $\Gd_Y$ is a group 
which projects onto the group 
\begin{align*}
\Gd &= \left\{ \sA(x)=\sum_{n\geq 1} \sA_n\ x^n, 
\ \sA_1=1 \right\} 
\end{align*}
of usual formal diffeomorphisms on a line (tangent to the identity). 

The group $\Gd_Y$ contains two proper subgroups 
\begin{align*}
\Gr_Y &= \Gp\over \treeA 
= \big\{ \rho_\sA= \sum_{t\in Y} \sA_t \ \lvertexgraft{t}, 
\ \sA_{\treeO}=1 \big\}, \\  
\Gl_Y &= \treeA\under\Ge 
= \big\{ \lambda_\sA= \sum_{t\in Y} \sA_t \ \rvertexgraft{t}, 
\ \sA_{\treeO}=1 \big\}, 
\end{align*}
whose projection onto usual series is the product $x\,\Gi$, which is 
isomorphic to $\Gd$ itself. 

In \cite{Frabetti2008} it was proved that the group $\Ga$ is a proper 
subgroup of $\Gr_Y$ of the form 
\begin{align}
\label{Ga}
\Ga &= \left\{ \alpha_\sA=(\treeO-\treeA\under \sA)^{-1}\over \treeA,\ 
\sA\in \A[[Y]] \right\}, 
\end{align}
where $\A[[Y]]$ denotes the set of tree-expanded formal series 
$\sA=\sum_{t\in Y} \sA_t\ t$ with coefficients $\sA_t$ in the algebra $\A$. 
In fact, for any tree-expanded formal series 
$\sA=\sum_{t\in Y} \sA_t\ t\in \A[[Y]]$, the series 
$\treeO-\treeA\under \sA$ belongs to the set $\Gi_Y$, 
and therefore to the group $\Gp$ of tree-expanded invertible series 
with respect to the product $\over$. 
The projection of $\Ga$ onto usual series is the group 
$x(1-x\A[[x]])^{-1}$ which coincides, again, with the whole group 
$\Gd$. 


\section{Tree-expanded series in relation with Tamari posets} 

In this section, we will consider some examples of tree-expanded
series in $\Gd_Y$. We will compute the inverse of some of them with
respect to composition and show that some others are related to the
Tamari lattices.

We will need the following involution on tree-expanded series. Let
$\sS=\sum_{n\geq 1} \sS_n$ be a tree-expanded series in $\Gd_Y$,
decomposed with respect to the order of the trees. The
\textit{suspension} of $\sS$ is the series $\tilde{\sS}=\sum_{n\geq 1}
(-1)^{n-1} \sS_n$. The suspension of a composition
$\widetilde{\sA\circ\sB}$ is the composition of the suspensions
$\tilde{\sA}\circ\tilde{\sB}$.

\bigskip

We will also use the following properties of the composition:
\begin{equation}
\forall \sA,\sC,\sD \quad  (\sA \over \sD)\circ\sC=(\sA\circ\sC) \over (\sD\circ\sC) \quad\text{and}\quad (\sA\under \sD)\circ\sC=(\sA\circ\sC) \under (\sD\circ\sC).
\end{equation}

These properties follow from the relation of $\Gd_Y$ with the
Duplicial operad.

\bigskip

Let us consider the tree-expanded series $\sA$ defined inductively by
\begin{equation}
  \sA=\treeA+w\, \sA \over \treeA + \treeA \under \sA + w \, \sA \over
  \treeA \under \sA,
\end{equation}
where $w$ is a formal parameter. This is in fact the sum over all
trees where each tree $t$ appears with coefficient $w$ to the power the
number of right-oriented leaves of $t$ minus $1$:
\begin{equation}
  \sA=\treeA+w\,\treeAB+\treeBA+w^2\,\treeABC+w\,\treeBAC+w\,\treeACA+w\,\treeCAB +\treeCBA+\cdots
\end{equation}

Let us consider also the tree-expanded series $\sB$ defined by
\begin{equation}
  \sB=\treeA-w \,\treeA \over \sB - \sB \under \treeA - w \, \treeA \over
  \sB \under \treeA,
\end{equation}
where $w$ is a formal parameter. This is in fact the sum over $p\geq
0$ and $q\geq 0$ of $(-1)^{p+q} w^p \, c_p \over \treeA \under d_q$
where $c_p$ is the left comb with $p$ vertices and $d_q$ is the right
comb with $q$ vertices:
\begin{equation}
  \sB=\treeA-w \treeAB-\treeBA+w^2\,\treeABC+w\,\treeACA +\treeCBA+ \cdots  
\end{equation}

\begin{proposition}
  In the group $\Gd_Y$, one has $\sA \circ\sB = \treeA$, that is $\sA$ is
  the inverse of $\sB$. This identity can be specialized to any
  complex value of $w$.
\end{proposition}

\begin{proof}
  Let us compute the composition $\sA \circ\sB$ starting from the definition of
  $\sA$. By properties of the composition, one gets, by multiplication on
  the right by $\sB$, that
  \begin{equation*}
    \sA \circ\sB=\sB+w\, (\sA\circ\sB) \over \sB + \sB \under (\sA\circ\sB) + w\, (\sA\circ\sB) \over  \sB \under (\sA\circ\sB).
  \end{equation*}
  Comparing this to the following form of the definition of $\sB$:
  \begin{equation}
    \treeA=\sB+w\, \treeA \over \sB + \sB \under \treeA + w\, \treeA \over
    \sB \under \treeA,
  \end{equation}
  one can see that $\treeA$ and $\sA \circ\sB$ both satisfy the same
  induction and have the same initial term, hence they are equal.
\end{proof}

Let us now consider the tree-expanded series $\sC$ and $\sD$ defined
by
\begin{equation}
  \sC={\treeA}+ \sC \over {\treeA}={\treeA}+ {\treeA} \over \sC={\treeA}+{\treeAB}+{\treeABC} +\cdots
\end{equation}
and 
\begin{equation}
  \sD={\treeA}+ \sD \under {\treeA}={\treeA}+ {\treeA} \under \sD={\treeA}+{\treeBA}+{\treeCBA}+\cdots
\end{equation}
So $\sC$ is in fact the sum over all left combs and $\sD$ is the sum
over all right combs. The series $\sC$ (resp. $\sD$) belongs to the
subgroup of $\Gd_Y$ formed by series indexed by left combs only (resp.
right combs only). It is quite clear (by inversion in these subgroups)
that $\sC^{-1}=\tilde{\sC}$ and $\sD^{-1}=\tilde{\sD}$.

\begin{proposition}
  One has
  \begin{equation}
    \label{Varbre}
    \sC + \sC \under \sD = \sD + \sC \over \sD.
  \end{equation}
  The series $\sE=\sC \circ\sD^{-1}$ is characterized by the equation
  \begin{equation}
    \sE={\treeA} + \sE \over {\treeA} - \sE \under {\treeA}.
  \end{equation}
\end{proposition}
\begin{proof}
  It is enough to see that the left hand side of (\ref{Varbre}) is
  exactly the sum of all trees of the shape $c^p \over {\treeA} \under
  d^q$ where $c_p$ is a left comb (possibly empty) and $d_q$ is a
  right comb (idem). By symmetry, this is also true for the right-hand
  side, hence they are equal. The equation for $\sC \circ\sD^{-1}$ is
  obtained from this by composition on the right by $\sD^{-1}$.
\end{proof}

The series $\sE$ appears surprisingly to be related to the M{\"o}bius
function of the Tamari lattice. Recall that the Tamari lattice is a
classical partial order \cite{Tamari} on the set $Y_n$ of trees of
order $n$, where the left comb $c_n$ is the unique minimal element and
the right comb $d_n$ is the unique maximal element. The partial order
is defined by transitive closure of the following relation: $t\leq t'$
if $t'$ is obtained from $t$ by replacing a local configuration
$\treeAB$ by a local configuration $\treeBA$.

\begin{proposition}
  One has the following description of the series $\sE$ and $\sD^{-1}$:
  \begin{equation}
    \sE=\sum_{n \geq 1} \sum_{t\in Y_n} \mu(c_n,t) t
  \end{equation}
  and
  \begin{equation}
    \sD^{-1}={\treeA}+\sum_{n \geq 1} \sum_{t\in Y_n} \mu(c_n,\treeA
    \under t) {\treeA} \under t,
  \end{equation}
  where $\mu$ is the M{\"o}bius function of the Tamari lattice.
\end{proposition}

\begin{proof}
  The first statement is equivalent to the second one, by using the
  following property of Tamari lattices, see for instance \cite[Lemma
  2.1]{BjWa2}: any interval $[c_{p+q},t' \over t'']$ in a Tamari lattice is
  isomorphic to the product of the intervals $[c_{p},t']$ and
  $[c_{q},t'']$ in smaller Tamari lattices. This shows that the
  generating series of M{\"o}bius number for all trees is the composition of
  $\sC$ by the generating series of M{\"o}bius number for trees of the
  shape $\treeA \under t$.

  The second statement can be deduced from the computation of the
  M{\"o}bius function of the Tamari lattice by Bj{\"o}rner and Wachs:
  see \cite[Corollary 9.5]{BjWa2}. In their notations, $c_n$ corresponds to
  the word $(0,\dots,0)$ and a tree of shape $\treeA \under t$ to a
  word $w_t$ beginning with the letter $n-1$. One gets that the
  M{\"o}bius number $\mu((0,\dots,0),w_t)$ is non zero if and only if
  the word $w_t$ is $(n-1,n-2,\dots,2,1)$, which corresponds to the
  tree $t=d_n$ in our notations.
\end{proof}

Remark: By using the projection morphism from $\Gd_Y$ to $\Gd$, one can see
that the sum of coefficients in $\sE$ of all trees of a fixed order
$n>1$ is zero, which also follows from the proposition.

\medskip

Let us then introduce the tree-expanded series $\sR={\treeA}+{\treeA}
\under \sA$ and $\sL={\treeA}+\sA \over {\treeA}$. In fact, $\sR$ is
the sum of all trees of the shape ${\treeA}\under t$ and $\sL$ is the
sum of all trees of the shape $t \over {\treeA}$. The series $\sR$
(resp. $\sL$) belongs to the subgroup of $\Gd_Y$ formed by series
indexed by trees of the shape ${\treeA}\under t$ only (resp. by trees
of the shape $t \over {\treeA}$ only).

\begin{proposition}
  One has
  \begin{align}
    \sR&={\treeA}+\sR \under \sL,\\
    \sL&={\treeA}+\sR \over \sL.
  \end{align}
  The composition $\sR \circ\sL^{-1}$ is equal to the suspension of $\sE$.
\end{proposition}

\begin{proof}
  The first two formulas follows directly from a standard
  combinatorial argument using a decomposition of trees. For instance,
  the first formula can be deduced from the existence of an unique
  maximal decomposition of a tree as an iterated $\under$ product. 

  By multiplication on the right by $\sL^{-1}$, one has
  \begin{align}
    \sR\circ\sL^{-1}&=\sL^{-1}+(\sR \circ\sL^{-1})\under {\treeA},\\
    {\treeA}&=\sL^{-1}+(\sR \circ\sL^{-1})\over {\treeA}.
  \end{align}  
  Hence by elimination of $\sL^{-1}$ one has
  \begin{equation}
 \sR\circ\sL^{-1}={\treeA}+(\sR \circ\sL^{-1})\under {\treeA}-(\sR \circ\sL^{-1})\over {\treeA}.
  \end{equation}
  By definition of the suspension, the suspension of $\sE$ and $\sR\circ
  \sL^{-1}$ satisfy the same induction, hence they are equal.
\end{proof}

Remark: one can also deduce from this a similar description of
$\sL^{-1}$.


\bibliographystyle{alpha}
\bibliography{bonn2007}

\end{document}